\documentclass[12pt]{article}
\usepackage{amssymb,amsmath,amsthm,geometry,amsfonts,mathrsfs,amscd}
\usepackage[utf8]{inputenc}
\usepackage{setspace, parskip}
\usepackage[small,nohug,heads=vee]{diagrams}
\diagramstyle[labelstyle=\scriptstyle]
\newtheorem*{mydef}{Definition}
\newtheorem{thm}{Theorem}[section]

\newtheorem*{rmk}{Remark}
\newtheorem{rmrk}{Remark}
\newtheorem{prop}{Proposition}[section]
\newtheorem*{thmx}{Theorem}

\newcommand{\longisomto}{\overset{\sim}{\longrightarrow}}
\numberwithin{equation}{section}

\title{A UNIVERSAL FORMULA FOR THE j-INVARIANT OF THE CANONICAL LIFTING}
\author{Altan Erdo\u{g}an\thanks{Sponsored by T\"{U}B\.{I}TAK (an agency of Turkish government)}}
\begin{document}
\date{}
\maketitle{}
\begin{abstract}
We study the $j$-invariant of the canonical lifting of an elliptic curve as a Witt vector. We prove that its Witt coordinates lie in an open affine 
subset of the $j$-line and deduce the existence of a universal formula for the $j$-invariant of the canonical lifting.
The canonical lifting of the elliptic curves with 
$j$-invariant 0 and 1728 over any characteristic are also explicitly found.
\end{abstract}

\renewcommand{\thefootnote}{\fnsymbol{footnote}} 
\renewcommand{\thefootnote}{\arabic{footnote}} 

\section{Introduction}
  Let $k$ be an algebraically closed field of characteristic $p$ and $W(k)$ be the ring of $p$-typical Witt vectors of $k$. Let $E$ be an ordinary elliptic curve over $k$. 
A consequence of the Serre-Tate theorem is that up to isomorphism there exists a unique elliptic curve $\mathbb{E}$ over $W(k)$ 
satisfying some certain conditions whose reduction modulo $p$ is isomorphic to $E$. Throughout the paper we will give the definition and other properties of the canonical lifting.
By definition, the $j$-invariant of $\mathbb{E}$, denoted by $j(\mathbb{E}) \in W(k)$ depends only on the $j$-invariant of $E$, say $j_0$. If we set 
$k^{\text{ord}}=\{j_0 \in k | \text{ elliptic curves with } j\text{-invariant }j_0\text{ are ordinary}\}$, we can define the following function;
\begin{eqnarray}
\Theta: k^{\text{ord}} &\longrightarrow& W(k), \nonumber \\
j_0 &\longmapsto& j(\mathbb{E})=(j_0,j_1,...) \nonumber
\end{eqnarray}
where $\mathbb{E}$ is the canonical lifting of $E$ and each $j_i$ is a function of $j_0$. 
The question of finding the canonical lifting in this form was first given in \cite{serretatelubin}. Here we will prove the following theorem.
\begin{thm}[Main theorem]
\label{mainthm}
 Let $F$ be a perfect field of characteristic $p>0$ with a fixed algebraic closure $k$, and $J$ be an indeterminate. Let $\phi_p(J)$ denote the \emph{Hasse polynomial}, i.e. the 
polynomial in $\mathbb{F}_p[J]$ whose roots are the supersingular $j$-values in characteristic $p$. Let
\begin{eqnarray}
 A=F[J,1/\phi_p(J)], \hspace{0.5cm} \nonumber
\end{eqnarray}
\begin{itemize}
 \item[\emph{\textbf{(i)}}]
There exist $f_i \in A$ for all $i \in \mathbb{Z}_{\geq1}$ such that for any $j_0 \in k^{\emph{ord}}$, 
\begin{eqnarray}
\Theta(j_0)=(j_0, f_1(j_0), f_2(j_0),...f_n(j_0),...) \nonumber
\end{eqnarray}
where for any such $j_0$ we see $f_i \in A$ as the homomorphism defined as
\begin{eqnarray}
 j_i:A \longrightarrow k, \hspace{0.5cm} J \longmapsto j_0 \nonumber
\end{eqnarray}
\item[\emph{\textbf{(ii)}}] If $j_0=0$ is an ordinary $j$-value then $\Theta(0)=0 \in \mathbb{Z}_p$, and similarly if $j_0=1728$ is an ordinary $j$-value then $\Theta(1728)=1728 \in 
\mathbb{Z}_p$.
\end{itemize}
\end{thm}
  The first assertion of the theorem mean that we have a universal formula for the $j$-invariant of the canonical lifting, and Witt entries of this universal 
formula are almost polynomials. The second assertion is independent of the previous one and proved with a different argument. 

  We proceed as follows. In \textsection{2} we give a brief overview of the Serre-Tate theorem. In \textsection{3}, we generalize the notion of the 
canonical lifting for elliptic curves defined over $\mathbb{F}_p$-schemes satisfying certain hypotheses. 
In \textsection{4} we use fppf-Kummer theory to prove that
the canonical lifting of an ordinary elliptic curve over an imperfect field is defined over the Witt ring of this imperfect field. 
This result first appeared in 
\cite{finotti}, but here we give a different proof. Finally in the last section we apply the results of \textsection{3} and \textsection{4} to 
universal families of ordinary elliptic curves to prove (i) of Theorem \ref{mainthm}. We also prove (ii) in the last section.

  We fix the following notation. For any schemes $X/T$ and 
$U/T$ we set $X_U := X \times_{T} U$. If $U=\,$Spec$\,C$ is affine, we may use $X_{C}$ instead of $X_{U}$. If $T=\,$Spec$\,B$ is also affine, we may also use $X \otimes_{B} C$ for $X_U$.
For $t \in T$ with residue field $\kappa(t)$, we denote $X \times_{T} $Spec$\,\kappa(t)$ by $X_t$. For any group scheme $G/T$, $G[N]$ denotes the kernel
of the multiplication by $N$ on $G$. If $G$ is a $p$-divisible group then we may write $G=(G_n,i_n)$ where $G_n=$ker$(p^n:G_{n+1} \longrightarrow G_{n+1})$ and 
$i_n: G_n \longrightarrow G_{n+1}$.

\subsection*{Acknowledgments} I would like to thank my advisor Sinan \"{U}nver for his invaluable support and comments at every step of my PhD research and this paper and Brian Conrad 
for his invaluable suggestions regarding this paper.

\section{An overview of the Serre-Tate theorem}
  
  In this section we briefly recall some aspects of the Serre-Tate theorem. We restrict ourselves to the definition-construction of the 
canonical lifting which is directly used in the proofs. General references for a complete proof and a detailed analysis of the Serre-Tate theorem are \cite{Katz} and 
\cite{Messing}. For the sake of completeness we quote the following theorem from \cite{Katz} and call it as the \emph{general Serre-Tate theorem}.

\begin{thm}[General Serre-Tate theorem]
Let $A$ be a ring in which $p$ is nilpotent. Let $I$ be a nilpotent ideal of $A$, and put $A_0=A/I$. Let 
$\emph{AS}(A)$ denote the category of abelian schemes over $A$, and let $\emph{Def}(A,A_0)$ denote the category
of triples $(X_0,L,\epsilon)$ where $X_0$ is an abelian scheme over $A_0$, $L$ is a $p$-divisible group over $A$ 
and $\epsilon: L_0:=L \otimes_{A} A_0 \longrightarrow A_{0}[p^{\infty}]$ is an isomorphism. Then the functor
\begin{eqnarray}
X \longmapsto (X_0, X[p^{\infty}], \emph{the natural map}) \nonumber
\end{eqnarray}
is an equivalence of the categories $\emph{AS}(A)$ and $\emph{Def}(A,A_0)$.
\end{thm}

  Let $k$ be an algebraically closed field of characteristic $p>0$ and $A$ be an Artin local ring
with residue field $k$. In general we say that a lifting of a scheme $X \longrightarrow \,$Spec$\,k$ is a pair $(\mathbb{X}, \iota)$ where $\mathbb{X} \longrightarrow \,$Spec$\,A$ is 
a scheme over Spec$\,A$ and $\iota: \mathbb{X} \otimes_{A} k \longisomto X$ is an isomorphism. If $\iota$ is unique we omit it and just say that 
$\mathbb{X} \longrightarrow \,$Spec$\,A$ is a lifting of $X \longrightarrow \,$Spec$\,k$. We can replace Spec$\,A$ by any scheme with some residue field $k$ and still can define a 
lifting in a similar way, but for our purposes we only consider lifting over Artin local rings.
Given an ordinary abelian variety $X$ over $k$, the Serre-Tate theorem classifies all abelian
schemes defined over $A$ that lift $X$.  For such an ordinary abelian variety $X \longrightarrow \,$Spec$\,k$ and an abelian scheme 
$\mathbb{X} \longrightarrow \,$Spec$\,A$ lifting $X/k$, there are the associated $p$-divisible groups (= Barsotti-Tate groups) 
denoted by $X[p^{\infty}]$ and $\mathbb{X}[p^{\infty}]$
respectively which play an important role summarized in the following diagram;
\begin{eqnarray}
\{\text{Isomorphism classes of } \mathbb{X}/A \text{ lifting } X/k \} &\longisomto& \nonumber \\
\{\text{Isomorphism classes of } \mathbb{X}[p^{\infty}]/A \text{ lifting } X[p^{\infty}]/k \} &\longisomto& \nonumber \\
\text{Ext}_{A}(T_{p}(X)(k) \otimes \mathbb{Q}_{p}/\mathbb{Z}_{p}, \text{Hom}_{\mathbb{Z}_{p}}(T_{p}(X^D)(k),\hat{\mathbb{G}}_{m}) &\longisomto& \nonumber \\
\text{Hom}_{\mathbb{Z}_{p}}(T_{p}(X)(k) \otimes T_{p}(X^D)(k),\hat{\mathbb{G}}_{m}(A)), \nonumber
\end{eqnarray}
where $T_{p}(X)(k)$ is the Tate module of $X$, $X^D$ denotes the dual abelian variety, $\hat{\mathbb{G}}_{m}$
denotes the formal completion of the multiplicative group $\mathbb{G}_{m}$ and Ext$_{A}(-,-)$ denotes the extension
group of $A$-groups. General references for the properties of $p$-divisible groups are \cite{Tate2} and \cite{Messing}.
The above diagram shows that the set
\[
\{\text{Isomorphism classes of } \mathbb{X}/A \text{ lifting } X/k \}
\]
has a natural group structure. 
\begin{mydef}
With the above notation the unique abelian scheme $\mathbb{X}/A$ which corresponds to the identity element of the group
\begin{eqnarray}
\emph{Ext}_{A}(T_{p}(X)(k) \otimes \mathbb{Q}_{p}/\mathbb{Z}_{p}, \emph{Hom}_{\mathbb{Z}_{p}}(T_{p}(X^D)(k),\hat{\mathbb{G}}_{m})) \nonumber
\end{eqnarray}
is called the canonical lifting of $X/k$ over $A$.
\end{mydef}
\begin{rmrk}
\emph{In the introduction, the base of the canonical lifting is given to be a characteristic zero integral domain, but here we define it over an Artin local ring which indeed 
fits well to our purposes. At the end of this section we will see that this definition is justified.
}
\end{rmrk}
\begin{rmrk}
\label{perfectbases}
\emph{If we only assume that $k$ is perfect than the above diagram and the definition
still remain valid by a slight change of the objects involved.  
A complete study of equivalent definitions of the canonical lifting 
for perfect $k$ can be found in \cite[V.3 and the Appendix]{Messing}.}
\end{rmrk}

The particular case we are concerned with here is the case where $A=W_{n}(k)$, the ring of $p$-typical Witt vectors of length $n$. 
Recall that if $k$ is a perfect field of characteristic $p$ then $W_{n}(k)$ is an Artin local ring with residue field $k$ and maximal ideal $(p)$. 
See \cite{serrelocalfields} for the definition and basic facts about Witt vectors which we use here. In this case the canonical liftings $\mathbb{X}_{m}/W_{m}(k)$ are 
compatible with each other, i.e. for any $m \leq n$,
\[
\mathbb{X}_{n} \otimes_{W_{n}(k)} W_{m}(k) \longisomto \mathbb{X}_{m}.
\]
Thus the inverse system $(\mathbb{X}_{m}/W_{m}(k))_m$ defines a formal abelian scheme over $W(k)$ which can be 
algebraicized (\cite{Messing}, \textsection{V}.3.3). This abelian scheme is defined as the canonical lifting of $X/k$ over $W(k)$, and hence justifies our definition above.
If there is no confusion about the base we will just say the canonical lifting of $X/k$.

\section{Canonical lifting of families}
\label{Integrality}
  In this section we will show that we can extend the definition of the canonical lifting to elliptic curves defined over $\mathbb{F}_p$-schemes under some hypothesis. This will allow
us to mention about the canonical lifting of a family of elliptic curves. Main result of this section is Theorem \ref{mainthmB} which is stated and proved at the end of this section.
  
  We fix the following notation for this section. Let $F$ be a perfect field of characteristic $p$ and $R$ be a Noetherian integral $F$-algebra with fields of fractions $K$. 
We fix an algebraic closure of $K$, and denote it by $\bar{K}$. Let $K'$ be the perfect closure of $K$ (i.e. the maximal purely inseparable extension of $K$) in $\bar{K}$ and $R'$ 
be the integral closure of $R$ in $K'$. We also define the subrings
\begin{eqnarray}
R_{n}=R^{1/p^n}=\{x \in \bar{K} | x^{p^n} \in R\}. \nonumber
\end{eqnarray}

Note that $R_n$ is Noetherian, and $R'=\cup_{n} R_{n}$ and the morphism of schemes Spec$\,R' \longrightarrow$ Spec$\,R$ induced by the inclusion $R \hookrightarrow R'$ is a 
homeomorphism. If $s'\in \,$Spec$\,R'$ maps to $s\in \,$Spec$\,R$ then $\kappa(s')$ is the perfect closure of $\kappa(s)$ \cite{greenbergperfect}.
Let $E/R$ be an ordinary elliptic curve in the sense of \cite[\textsection{2} and \textsection{12}]{KM}. 
Let $E_n := E \otimes_{R} R_n$ where the base change is done via the $p^{n}$-th root homomorphism $R \longrightarrow R_n$.
Throughout this section $E/R$ and $E_n/R_n$ will always denote these elliptic curves defined here. To simplify notation we use $E$ also to denote the base extensions 
$E \otimes_{R,i} R_n$ and $E \otimes_{R,i}R'$ where $i$ is the inclusion map.

  Now let $T$ be the spectrum of a complete Noetherian local ring. Then for any finite locally free group scheme $G/T$, there is a unique exact sequence 
called the \emph{connected- \'etale sequence} of $G$,
\[
0 \longrightarrow G^0 \longrightarrow G \longrightarrow G^{et} \longrightarrow 0
\]
where $G^{0}$ and $G^{et}$ are connected and \'etale $T$-group schemes respectively. It is characterized by the fact that for any \'etale $T$-group $H$, any $T$-group homomorphism
$G \longrightarrow H$ factors through $ G \longrightarrow G^{et}$ \cite{Tate1}. By passage to limit we have 
a similar construction for $p$-divisible groups. If $G=(G_n,i_n)_n$ is a $p$-divisible group, then $G^{0}:=(G_n^{0}, i_n)$ and 
$G^{et}:=(G_n^{et}, i_n)$ are connected and  \'etale $p$-divisible groups respectively. By \cite{Tate2} we have an exact sequence of $p$-divisible groups
\[
 0 \longrightarrow G^0 \longrightarrow G \longrightarrow G^{et} \longrightarrow 0.
\]
Note that if $T=$Spec$\,F$ then the connected-\'etale sequence of $G$ splits. In particular if $\tilde{E}$ is an ordinary elliptic curve over $F$ then the connected-\'etale sequence
of $\tilde{E}[p^{\infty}]$ splits over $F$.
This fact is very crucial in the construction of the canonical lifting (Recall Remark \ref{perfectbases}).
So in order to generalize the notion of the canonical lifting we need a similar result 
when $F$ is replaced by another scheme. But \emph{a priori} we don't even know that $E[p^{\infty}]$ has a connected-\'etale sequence over an arbitrary base. 
The following theorems of Messing which we directly quote from \cite{Messing} and the Proposition \ref{splitting} below allow us to overcome this problem. 

\begin{thm}
\label{uniqueness}
Let $S$ be any scheme and $f:\,X \longrightarrow S$ be a finite locally free morphism of schemes. 
Then the function $s \longmapsto$ (separable rank of $X_{s}$) is locally constant on $S$ if and only if there are morphisms 
$i:\,X \longrightarrow X'$ and $f':\,X' \longrightarrow S$ which are finite and locally free with $i$ radiciel and surjective, $f'$  \'etale
and $f = f' \circ i$. The factorization is unique up to unique isomorphism and is functorial in $X/S$.
\end{thm}
\begin{proof}
\cite[\textsection{II}.4.8]{Messing}.
\end{proof}
\begin{thm}
\label{existence}
Let $S$ be a scheme on which $p$ is locally nilpotent, and $G$ be a $p$-divisible group over $S$. Then the 
following conditions are equivalent.
\begin{enumerate}
\item[\emph{\textbf{(i)}}] $G$ is an extension of an \'etale $p$-divisible group by a connected $p$-divisible group.
\item[\emph{\textbf{(ii)}}] The function $s \longmapsto$ (separable rank of $G[p]_{s}$) is locally constant on $S$.
\end{enumerate}
\end{thm}
\begin{proof}
We only take the relevant parts of \cite[\textsection{II}.4.9]{Messing}.
\end{proof}
 
\begin{prop}
\label{splitting}
Let $E[p^{n}]$ denote the kernel of $p^n:E \longrightarrow E$ and $E[p^{\infty}]$ be the $p$-divisible group of $E$. 
\begin{itemize} 
 \item[\textbf{\emph{(i)}}] For each $n$ there is a unique connected- \'etale sequence
\begin{eqnarray}
\label{connectedetaletorsion}
0 \longrightarrow E[p^n]^{0} \longrightarrow E[p^n] \longrightarrow E[p^n]^{et} \longrightarrow 0 \nonumber
\end{eqnarray}
which splits over $R_n$.
\item[\textbf{\emph{(ii)}}] There is a unique connected- \'etale sequence of $p$-divisible groups 
\begin{eqnarray}
\label{connectedetaleinfty}
0 \longrightarrow E[p^{\infty}]^{0} \longrightarrow E[p^{\infty}] \longrightarrow E[p^{\infty}]^{et} \longrightarrow 0, \nonumber 
\end{eqnarray}
which splits over $R'$.
\end{itemize}
\end{prop}

\begin{proof}
By hypothesis $E$ is ordinary, so the $p$-divisible group $G=E[p^{\infty}]$ satisfy the last condition of Theorem \ref{existence}, and the existence of both sequences follow. Recall 
the notation we adopted for $E$, i.e. we can take the base to be $R$, $R_n$ or $R'$ and so we have the relevant connected- \'etale sequences over any of these bases. But by the 
uniqueness assertion of Theorem \ref{uniqueness} these sequences are compatible with each other in the sense that $E_{R'}[p^n]^{0} = (E[p^n]^{0})_{R'}$ and 
similarly for other groups and the other base $R_n$. Thus the uniqueness of the sequences in the theorem also follow. 

  The remaining thing is to prove the splitting of the sequences over the specified bases. First note that the splitting of the sequences given in (i) for all $n$ imply the splitting 
of the sequence given in (ii). Thus we only need to show that 
\[
 0 \longrightarrow E[p^n]^{0} \longrightarrow E[p^n] \longrightarrow E[p^n]^{et} \longrightarrow 0 
\]
splits over $R_n$. Also the groups involved here are all commutative, so 
splitting amounts to giving a section of $E[p^n] \longrightarrow E[p^n]^{et}$ over $R_n$. 

  Let $F^n:\,$Spec$\,R_n \longrightarrow$Spec$\,R_n$ be the $n$-th iterate of the absolute Frobenius of Spec$\,R_n$. Then we have
\begin{eqnarray}
E_n^{(p^n)}:= E_n \otimes_{R_n,F^{n}} R_n = E \otimes_{R,i} R_n (=E). \nonumber 
\end{eqnarray}
To simplify notation we use $F^n$ also 
to denote the $n$-th iterate of the relative Frobenius of $E_n$; $F^n: E_n \longrightarrow E_{n}^{(p^n)}$. We denote the dual isogeny of $F^n$ by $V^n$. 
Then we have the following commutative diagram
\begin{diagram}
E &\rTo^{V^{n}} &E_{n}\\
&\rdTo_{[p^n]} &\dTo_{F^{n}}\\
& & E
\end{diagram}
which shows that ker$(V^n)$ is a subgroup of $E[p^n]$. But since $E$ is ordinary then ker$(V^n)$ is a finite  \'etale group over $R_n$ \cite[\textsection{12}.3.6]{KM}. 
The inclusion ker$(V^n) \hookrightarrow E[p^n]$ will give a required section once we can show that ker$(V^n) \longisomto E[p^n]^{et}$. Note that this isomorphism holds 
over algebraically closed fields; since then $E[p^n] = \mu_{p^n} \times \mathbb{Z}/p^{n}\mathbb{Z}$, and $V^n$ is the identity on $\mu_{p^n}$ and 
kills $\mathbb{Z}/p^{n}\mathbb{Z}$. Our aim is to reduce to this case. In general the composition
\begin{eqnarray}
 \text{ker}(V^n) \longrightarrow E[p^n] \longrightarrow E[p^n]^{et}
\end{eqnarray}
is a group homomorphism, so necessarily commutes with the action of the  \'etale fundamental group (see \cite{Tate1} for the definition of the  \'etale fundamental group). Finally 
we remark that $R_n$ is Noetherian for any $n$. Now the following theorem of Grothendieck completes the proof.
\end{proof}
\begin{thmx}
Let $S$ be a locally Noetherian scheme, $\alpha$ be a geometric point, and $\pi=\pi(S,\alpha)$ be 
the  \'etale fundamental group of $S$ centered at $\alpha$. Then the functor $Y \longmapsto Y(\alpha)$
establishes an equivalence between the category of finite  \'etale schemes over $S$ and the category of finite 
sets with a continuous $\pi$ action.
\end{thmx}
\begin{rmk}
\emph{We need Noetherian hypothesis only in the last step of the proof, i.e. only to use this theorem of Grothendieck. So the sequences in (i) and (ii) still exist if we drop the Noetherian 
condition on $R$.}
\end{rmk}
\begin{rmk}
\emph{The sequence given in (ii) of Proposition \ref{splitting} also captures the connected- \'etale sequence of the fibre $\kappa(s')$ for any $s' \in \,$Spec$\,R'$ in the 
following sense;}
\begin{eqnarray}
(E[p^{\infty}]^{et})_{s'}  = (E_{s'}[p^{\infty}])^{et} \text{  }\emph{  and  } \text{  } (E[p^{\infty}]^{0})_{s'}=(E_{'s}[p^{\infty}])^{0}. \nonumber
\end{eqnarray}
\end{rmk}

  Now we will use Proposition \ref{splitting}, and the general Serre-Tate theorem
to find a good lifting of $E(=E_{R'})$ to $W_{m}(R')$ for each $m$. We 
will need the following important theorem of Grothendieck.
\begin{thm}
\label{grot2}
Let $A$ be a ring, $I$ an ideal of $A$. Suppose that $A$ is complete and separated with respect to topology 
defined by the ideal $I$. Put $A_{0}=A/I$. Then the functor
\[
X \longmapsto X \otimes_{A} A_0 
\]
establishes an equivalence between the category of finite  \'etale $A$-schemes and the category of finite  \'etale
$A_0$-schemes.
\end{thm}
\begin{proof}
\cite[\textsection{18}.3.2]{EGA4}.
\end{proof}
\begin{thm}
\label{mainthmB}
Let $R$ be a Noetherian, integral $F$-algebra with perfect closure $R'$, and $E$ be an ordinary elliptic curve over
$R$. Then for each $m$ there exists a unique elliptic curve $\mathbb{E}_m/W_m(R')$ lifting $E/R'$
such that the $p$-divisible group $\mathbb{E}_m[p^{\infty}]$ has a split exact connected-\'etale sequence. Moreover for any $s'\in \,$\emph{Spec}$\,R'$ with residue field $\kappa(s')$, 
the elliptic curve
\[
 \mathbb{E}_m \otimes_{W_m(R')} W_m(\kappa(s'))
\]
is the canonical lifting of $E_{s'}$ over $W_m(\kappa(s'))$.
\end{thm}

\begin{proof}
Let $m \geq 2$ be a fixed integer. Since $R'$ is a perfect ring, $A=W_m(R')$ satisfies the hypothesis of Theorem \ref{grot2}. 
So for any $n$ there exists a unique \'etale group scheme $H_n$ over $W_{m}(R')$ such that
\[
H_n \otimes_{W_{m}(R')} R' \longisomto E[p^n]^{et}.
\]
It also follows that $H_n$ form an inductive system, and so the limit gives a $p$-divisible group $H_{\infty}$ lifting 
$E[p^{\infty}]^{et}$.
Applying Cartier duality to $E[p^{n}]^{et}$, we see that $E[p^{n}]^{0}$ and so $E[p^{\infty}]^{0}$ has also a unique lifting $G_{\infty}$ to 
$W_{m}(R')$. Since the sequence given in (ii) of Proposition \ref{splitting} is split exact, the product $G_{\infty} \times H_{\infty}$ lifts the 
$p$-divisible group $E[p^{\infty}]$. 

  By the general Serre-Tate theorem there is a unique abelian scheme $\mathbb{E}_m$ over $W_m(R')$ lifting $E$ which corresponds to $G_{\infty} \times H_{\infty}$, 
and so has a split exact connected- \'etale sequence
\begin{eqnarray}
0 \longrightarrow \mathbb{E}_{m}[p^{\infty}]^{0} \longrightarrow \mathbb{E}_{m}[p^{\infty}] \longrightarrow \mathbb{E}_{m}[p^{\infty}]^{0} \longrightarrow 0. \nonumber 
\end{eqnarray}
By checking fibers or dimension we can see that $\mathbb{E}_m$ is indeed an elliptic curve. 

  Now by construction for any $s' \in \,$Spec$\,R'$, the elliptic curve $\mathbb{E}_m \otimes_{W_{m}(R')} W_{m}(\kappa(s'))$ lifts $E_{s'}$ and the associated $p$-divisible group 
\[
 \mathbb{E}_{m}[p^{\infty}] \otimes_{W_{m}(R')} W_{m}(\kappa(s'))
\]
has a split exact connected- \'etale sequence. So $\mathbb{E}_m \otimes_{W_{m}(R')} W_{m}(\kappa(s'))$ must be the canonical lifting of $E_{s'}$.
\end{proof}

  For the purposes of this paper we may call the elliptic curve $\mathbb{E}_m$ as the canonical lifting of $E$ over $W_m(R')$. The
$j$-invariant of $\mathbb{E}_m$, denoted by $j(\mathbb{E}_m)$ will be the universal formula for the canonical lifting of the fibers 
in the following sense: Let 
\[
 j(\mathbb{E}_m)=(j_0,j_1,...,j_{m-1})
\]
and let $f_{s'}:\, R' \longrightarrow \kappa(s')$ be the canonical map for $s'\in \,$Spec$\,R'$. Then the $j$-invariant of the canonical lifting of $E_{s'}$ over $W_m(\kappa(s'))$
is given by
\[
 j=(f_{s'}(j_0),f_s(j_1),...,f_s(j_{m-1})).
\]

\section{Canonical lifting over imperfect fields }
  In this section we will prove that the base of the canonical lifting has a well behavior with respect to the base of 
the given ordinary elliptic curve. Explicitly we will prove the following theorem.
\begin{thm}
\label{descent_to_sep}
Let $K$ be any field of characteristic $p > 0$, and let $E$ be an ordinary elliptic curve over $K$. Let $\mathbb{E}$ be the canonical lifting of $E$ over $W(\bar{K})$.
We denote the $j$-invariant of $\mathbb{E}$ by
$j(\mathbb{E})=(j_0,j_1,...,j_n,...)$. Then each $j_n$ is an element of $K$.
\end{thm}
\begin{proof}
 This theorem for $p \geq 5$ was proved by Finotti, L.R.A. in \cite{finotti} using Greenberg transforms and elliptic Teichm\"{u}ller lifts. Here we give a different proof. 
The theorem holds for perfect $K$ by definition of the canonical lifting. Let $K'$ and $K^{\text{sep}}$ denote the perfect and separable closures of $K$ respectively. Then we have 
$K' \cap K^s =K$. Since $j_n \in K'$ it suffices to show that $j_n \in K^{\text{sep}}$. Thus we may assume $K$ to be a separably closed field and $K'=\bar{K}$. 
In this case for any integer $n \geq 1$ we have the following isomorphisms over $K$,
\begin{eqnarray}
 E[p^{n}]^{0} &\longisomto& \mu_{p^{n}}, \nonumber \\
 E[p^{n}]^{et} &\longisomto& \mathbb{Z}/p^{n}\mathbb{Z} \nonumber
\end{eqnarray}
where $\mu_{p^{n}} = $ker$(p^{n}: \mathbb{G}_m \longrightarrow \mathbb{G}_m)$ and $\mathbb{Z}/p^{n}\mathbb{Z}$ is the Cartier dual of $\mu_{p^{n}}$. We may fix these isomorphisms 
to be compatible with Cartier duality. Then we have the exact sequence
\begin{eqnarray}
  0 \longrightarrow \mu \longrightarrow E[p^{\infty}] \longrightarrow \mathbb{Q}_p/\mathbb{Z}_p \longrightarrow 0 \nonumber
\end{eqnarray}
where $\mu$ and $\mathbb{Q}_p/\mathbb{Z}_p$ denote the p-divisible groups $(\mu_{p^{n}},i_n)$ and $(\mathbb{Z}/p^{n}\mathbb{Z},i_n)$ respectively.
By the general Serre-Tate theorem it suffices to show that there exists a $p$-divisible group $\mathbb{G}/W_m(K)$ lifting $E[p^{\infty}]$ given with an extension 
\begin{eqnarray}
  0 \longrightarrow \mu \longrightarrow \mathbb{G} \longrightarrow \mathbb{Q}_p/\mathbb{Z}_p \longrightarrow 0 \nonumber
\end{eqnarray}
which splits after base change to $W_m(K')$. So the result follows from the following more general theorem. 
\end{proof}

\begin{thm}
\label{mainimperfect}
Let $K$ be any field of characteristic $p > 0$ and $G=(G_n,i_{n})$ be a $p$-divisible group over $K$ given with an extension
\begin{eqnarray}
\label{overK}
  0 \longrightarrow \mu \longrightarrow G \longrightarrow \mathbb{Q}_p/\mathbb{Z}_p \longrightarrow 0.
\end{eqnarray}
Then there exists a $p$-divisible group $\mathbb{G}/W_m(K)$ lifting $G$ with an extension
\begin{eqnarray}
  0 \longrightarrow \mu \longrightarrow \mathbb{G} \longrightarrow \mathbb{Q}_p/\mathbb{Z}_p \longrightarrow 0 \nonumber.
\end{eqnarray}
which splits over $W_m(K')$. 
\end{thm}

Before we go into the proof we briefly give some facts which is crucial in the proof. Further details of this part can be completely found in \cite[8.7-10]{KM}.
We define a finite locally free group scheme $T[N]$ for any $N >0$ over $\mathbb{Z}[q,q^{-1}]$ as follows. As a scheme $T[N]$ is the disjoint union of 
\[
 T_i[N]=\text{Spec}\,(Z[q,q^{-1}][X]/(X^{N}-q^i))
\]
for $i=0,1,...,N-1$. For any connected $\mathbb{Z}[q,q^{-1}]$-algebra $C$ then we have
\[
 T[N](C)=\{(X,i/N) | X \in C, \, 0 \leq i  \leq N-1,\,X^N=q^i\}
\]
The group law is defined by
\[
(X,i/N).(Y,j/N) = \left\{
	\begin{array}{ll}
		(XY,(i+j)/N)  & \mbox{if } i+j \leq N-1, \\
		(XY/q,(i+j-N)/N) & \mbox{if } i+j \geq N.
	\end{array}
\right.
\]
It is easy to see that $T[N]$ is a finite locally free group scheme of order $N^2$ killed by $N$ and the elements of the form $(X,0)$ is a subgroup isomorphic to $\mu_N$.
So that we have an exact sequence
\[
 0 \longrightarrow \mu_{N} \longrightarrow T[N] \longrightarrow \mathbb{Z}/N\mathbb{Z} \longrightarrow 0.
\]
This exact sequence splits over $C$ if and only if the image of $q$ in $C$ has an $N$-th root. Indeed $T[N]$ is universal in the following sense.
\begin{prop}
\label{KMhelp}
 Let $S$ be any scheme and $G/S$ be a finite locally free group scheme over $S$ of order $N^2$ which is killed by $N$ given with an extension structure
\begin{eqnarray}
\label{basicG}
 0 \longrightarrow \mu_{N} \longrightarrow G \longrightarrow \mathbb{Z}/N\mathbb{Z} \longrightarrow 0.
\end{eqnarray}
Then Zariski locally on $S$ there exists $q\in \mathbb{G}_m(S)$ such that
\[
 G \longisomto T[N] \otimes_{\mathbb{Z}[q,q^{-1}]} S
\]
and this isomorphism is compatible with extension structures.
\end{prop}
\begin{proof}
 A complete proof can be found in \cite[8.10.5]{KM}. Here we do not give a complete proof. For our purposes we briefly recall the construction of the given isomorphism in the case
$S=$Spec$\,A$ is affine and Pic$(A)=0$. So that we may remove the Zariski local condition on the relevant isomorphism.	

  Now locally fppf (\ref{basicG}) splits. So $G$ is an fppf form of 
the product group scheme $\mu_{p^n} \times \mathbb{Z}/p^{n}\mathbb{Z}$. But the set isomorphism classes of fppf forms of $\mu_{p^n} \times \mathbb{Z}/p^{n}\mathbb{Z}$ is bijective 
to the set of isomorphism classes of Aut-torsors where Aut is the group scheme whose $T$-valued points are
\[
 \text{Aut}(T)=\left\lbrace \text{automorphisms of } \mu_{p^n} \times \mathbb{Z}/p^{n}\mathbb{Z} \text{ of the form } \left( \begin{array}{ccc}
1 & \phi  \\
0 & 1 \end{array} \right), \phi \in \text{Hom}_{T}(\mathbb{Z}/p^{n}\mathbb{Z}, \mu_{p^n}) \right\rbrace
\]
But the later set is just $H^{1}(\text{Spec}\,A,\text{Aut})$. Also Aut$\longisomto \mu_{p^n}$ and so the group $H^{1}(\text{Spec}\,A,\mu_{p^n})$ classifies the isomorphism classes 
of forms of $\mu_{p^n} \times \mathbb{Z}/p^{n}\mathbb{Z}$. Note that the torsor corresponding to $G$ is just the inverse image of 1 in $G \longrightarrow \mathbb{Z}/p^n\mathbb{Z}$.
Now consider the Kummer sequence
\[
 0 \longrightarrow \mu_{p^n} \longrightarrow \mathbb{G}_m \longrightarrow \mathbb{G}_m \longrightarrow 0.
\]
Since $H^{1}(\text{Spec}\,A, \mathbb{G}_m)=$Pic$(A)=0$ we have the following relevant part of the corresponding long exact sequence
\[
 \mathbb{G}_m(A) \xrightarrow{p^n} \mathbb{G}_m(A) \longrightarrow H^{1}(\text{Spec}\,A,\mu_{p^n}) \longrightarrow 0.
\]

  So for any cocycle in $H^{1}(\text{Spec}\,A,\mu_{p^n})$ the corresponding $\mu_{p^n}$-torsor is just $[p^n]^{-1}(q)$ for some $q \in A^{*}$. Note that $q$ is unique up to multiplying
by a $p^n$-th power in $A^{*}$. In particular the class of $G$ in $H^{1}(\text{Spec}\,A,\mu_{p^n})$ denoted by cl$(G)$ corresponds to a $\mu_{p^n}$-torsor 
$[p^n]^{-1}(q)$ for some $q \in A^*$. For any $A$-scheme $T$, the $T$-valued points of $[p^n]^{-1}(q)$ is the set 
$\{x \in \mathbb{G}_m(T): x^{p^n}=q\}$. Now consider the following extension of group schemes
\[
 0 \longrightarrow \mu_{p^n} \longrightarrow T[p^n] \otimes_{\mathbb{Z}[q,q^{-1}]} A \stackrel{\epsilon}{\longrightarrow} \mathbb{Z}/p^{n}\mathbb{Z} \longrightarrow 0.
\]
The $\mu_{p^n}$-torsor corresponding to $T[p^{n}]$ is then $\epsilon^{-1}(1)=[p^n]^{-1}(q)$, so that the images of $G$ and $T[p^n] \otimes_{\mathbb{Z}[q,q^{-1}]} A$ in 
$H^{1}(\text{Spec}\,A,\mu_{p^n})$ are the same and hence
\[
 G_n \longisomto T[p^n] \otimes_{\mathbb{Z}[q,q^{-1}]} A.
\]
\end{proof}

\begin{proof}[Proof of \ref{mainimperfect}]
 First note that any $p$-divisible group $\mathbb{G}$ lifting $G$ necessarily has an extension structure 
\begin{eqnarray}
  0 \longrightarrow \mu \longrightarrow \mathbb{G} \longrightarrow \mathbb{Q}_p/\mathbb{Z}_p \longrightarrow 0\nonumber.
\end{eqnarray}
This follows from Theorem \ref{existence}. So we only need to show the splitting. Now giving an extension as (\ref{overK}) is same as giving a compatible family of extensions
\begin{eqnarray}
\label{overKn}
0 \longrightarrow \mu_{p^n} \longrightarrow G_n \stackrel{\pi}{\longrightarrow} \mathbb{Z}/p^{n}\mathbb{Z} \longrightarrow 0 
\end{eqnarray}
for all $n$. Since $G_n$ satisfies the conditions of Proposition \ref{KMhelp} there exists $q \in K^*$ such that
\[
 G_n \longisomto T[p^{n}] \otimes_{\mathbb{Z}[q,q^{-1}]} K.
\]
By hypothesis we have that 
\[
G_{n-1}=\text{ker}(p^{n-1}:G_n \longrightarrow G_n) \longisomto T[p^{n-1}] \otimes_{\mathbb{Z}[q,q^{-1}]} K.
\]
This shows that the class of $G_{n-1}$ 
in $H^{1}(\text{Spec}\,K,\mu_{p^{n-1}})$ is the $\mu_{p^{n-1}}$-torsor $[p^{n-1}]^{-1}(q)$. So $G = (G_n,i_n)$ determines a non-unique sequence of elements $(q_n)$ in $K^*$ 
where 
$\overline{q_n} \in H^{1}(\text{Spec}\,K,\mu_{p^n}) \longisomto K^*/(K^*)^{p^n}$ and $\overline{q_n}=u_n^{p^{n-1}}\overline{q_{n-1}}$ for some $u_n \in K^*$, i.e. $G$ determines 
an element of the inverse limit
\[
 \varprojlim_{n} K^*/(K^*)^{p^n}.
\]
Conversely given an element $(\overline{q_n})\in \varprojlim_{n} K^*/(K^*)^{p^n}$ choose a sequence $(q_n)$ in $K^*$ such that $q_n \mapsto \overline{q_n}$. Then
we have that $q_n=u_n^{p^{n-1}}q_{n-1}$ and that $G=(T[p^n] \otimes_{\mathbb{Z}[q_n,q_n^{-1}]} A,i_n)_n$ is a $p$-divisible group. 
  
  Up to now we didn't use that $K$ is indeed a field, we only used that Pic$(K)=0$. We can carry out the same procedure to obtain a $p$-divisible group over $W_m(K)$, i.e. we 
need to specify a sequence $(Q_n)$ in $W_m(K)^{*}$ such that $Q_n=U_n^{p^{n-1}}Q_{n-1}$. Then the $p$-divisible group
\[
 \mathbb{G}=(T[p^n] \otimes_{\mathbb{Z}[Q_n,Q_n^{-1}]} W_m(K), i_n)
\]
is the $p$-divisible group corresponding to the chosen sequence $(Q_n)$. Now we impose the condition that $\mathbb{G}$ lifts $G$, i.e. $\mathbb{G}_n \otimes_{W_m(K)} K\longisomto G_n$
via the surjection $W_m(K) \longrightarrow K$ and the isomorphims are compatible with the maps $i_n: G_n \longrightarrow G_{n+1}$. This condition is satisfied if we choose $Q_n$ and 
$U_n$ such that $Q_n \mapsto q_n$ and $U_n \mapsto u_n$ under $W_m(K) \longrightarrow K$. Also we want the connected-\'etale sequence of $\mathbb{G}$ to split over $W_m(K')$. This 
means that $Q_n$ must have a $p^n$-th root in $W_m(K')$. All of these are satisfied if we set $Q_n$ and $U_n$ to be the Teichm\"{u}ller lifts of $q_n$ and $u_n$ respectively. 
Let
\begin{eqnarray}
 f: K^* &\longrightarrow& W_m(K)^{*}, \nonumber \\
 a &\longmapsto& (a,0,0,...,0) \nonumber
\end{eqnarray}
be the Teichm\"{u}ller map. 
Thus if we set $Q_n=f(q_n)=(q_n,0,0,...,0 )$ and $U_n=f(u_n)=(u_n,0,0,...,0)$, the corresponding $\mathbb{G}$ is the required $p$-divisible group. 
\end{proof}

\section{The universal formula}
Now we can use the results of the previous sections to prove Theorem \ref{mainthm} stated in \textsection{1}.
\begin{proof}[Proof of Theorem \ref{mainthm}]
Let $p$ be any prime number. Recall that we defined the ring $A$ as
\[
 A= F[J,1/\phi_p(J)].
\]
Let $K$ be the fields of fractions of $A$ with a fixed algebraic closure $\bar{K}$. We define the ring $R$ to be the localization of $A$ at the element $J(J-1728)$, i.e.
\begin{eqnarray}
R=F[J,1/J(J-1728)\Phi_p(J)]. \nonumber
\end{eqnarray}
We define the elliptic curve $E/R$ as
\begin{eqnarray}
 E: y^2+xy=x^3-36x/(J-1728)-1/(J-1728). \nonumber
\end{eqnarray}
Note that $j(E)=J$, $\Delta(E)=J^{2}/(J-1728)^2$ and $E$ is ordinary. So the hypotheses of Theorem \ref{mainthmB} is satisfied. With the same notation of Theorem \ref{mainthmB}, 
for any positive integer $m$ we have the elliptic curve 
$\mathbb{E}_m$ over $W_{m}(R')$ with $j$-invariant $j(\mathbb{E}_m)=(j_0,j_1,j_2,...,j_{m-1})$. Now for any $j_0 \in k^{\text{ord}} \setminus \{0,1728\}$, 
the homomorphism $R' \longrightarrow k$ induced by $J \longmapsto j_0$ maps $E$ to an ordinary 
elliptic curve over $k$ with $j$-invariant $j_0$, say $\tilde{E}$. Similarly it maps $\mathbb{E}_m$ to the canonical lifting of $\tilde{E}$. 
We set $f_i=j_i$ for all $i$. Now we show that $j_i \in R$. 

  Let $K$ and $K'$ be the fields of fractions of $R$ and $R'$ respectively. Consider the elliptic curve 
\begin{eqnarray}
 \mathbb{E}_m \otimes_{W_m(R')} {W_m(K')} \nonumber
\end{eqnarray}
obtained via the inclusion $W_m(R') \hookrightarrow W_m(K')$.
It is a lifting of the generic fiber of $E$ denoted by $E_{K'}=E \otimes_{R'} K'$ and has a split exact connected-\'etale sequence over $W_m(K')$. 
So it is the canonical lifting of $E_{K'}$.
Its $j$-invariant is in $W_{m}(R')$ as it is obtained from $\mathbb{E}_m$. 
But Theorem \ref{descent_to_sep} implies that the $j$-invariant of the canonical lifting of $E_{K}$ which is tautologically equal to the $j$-invariant of 
$\mathbb{E}_m \otimes_{W_m(R')} {W_m(K')}$ is indeed in $W_{m}(K)$. So each $j_i \in R' \cap K = R$ since $R$ is integrally closed. 
Now we will show that $j_i\in A$, i.e. $j_i$ are regular at 0 and 1728 provided that these are ordinary. 
We will also prove that $(0,j_1(0),j_2(0)...,)$ is the $j$-invariant of the canonical lifting of the elliptic curve with $j$-invariant zero, and similarly for 1728. 

  We may assume $p \geq 5$ bacause $0=1728$ is supersingular in characteristic 2 and 3. Let $\mu_3 \neq 1$ be a fixed cube root
of 1. Let $L=F(\mu_3,\sqrt3)$, $a=\sqrt[3]{J}$ and $b=2.3^{-1}.\sqrt{J-1728}/\sqrt{3}$. 
  
  Let $B=L[J, \sqrt[3]{J},\sqrt{J-1728},1/\phi_{p}(J)]$ and $E$ be the scheme over $B$ defined as
\[
 E':y^2=x^3+ax+b.
\]
Note that $B$ is an integral extension of $A$ and $E'$ is an elliptic curve over $B$ with $j(E')=J$ and $\Delta(E)=-4^3.1728 \neq 0$. Now $E'$ and $E$ are isomorphic over 
$\bar{K}$ since they have the same $j$-invariant. So the $j$-invariants of the canonical lifting of $E$ and $E'$ denoted by $\mathbb{E}_m$ and $\mathbb{E}_m'$ are the same. 
But by Theorem \ref{mainthmB} we have $j(\mathbb{E}') \in W_m(B')$ where $B'$ is the perfect closure of $B$. So $j(\mathbb{E}')=j(\mathbb{E})=(j_0,j_1,j_2,...,j_{m-1})$ where
$j_i \in B'$. But previously we proved that $j_i\in R$. Now $B'$ is an integral extension of $A$ and 
$A$ is integrally closed in $R$. Thus $R \cap B' =A$, and so $j_i \in A$. In particular $j_i$ are regular at any element of $k^{\text{ord}}$.

  Now let $j_0=0 \in k^{\text{ord}}$. Consider the maximal ideal $(J) \subset A$. Since $S'$ is an integral extension of $A$ there exists a prime ideal $q'$ of $S'$ such that 
$q' \cap A =(J)$. Integrality of $S'$ also implies that $q'$ is a maximal ideal. Let $\kappa=S'/q'$. Since $J\, (\text{mod } q')=0$ we have that
\[
 \mathbb{E}_m' \otimes_{W_m(S')} W_m(\kappa)
\]
is the canonical lifting of the elliptic curve with $j$-invariant zero. The $j$-invariant of this canonical lifting is just
\[
 (0,\overline{j_1}, \overline{j_2},...,\overline{j_{m-1}})
\]
where $\overline{j_i}=j_i \, (\text{mod } q')$. But $A/(J) \rightarrow S'/q'$ is injective and $j_i \in A$, so $j_i \, (\text{mod } q')$ is the image of $j_i \, (\text{mod } (J))$ 
under this injection. But 
$j_i \, (\text{mod } (J)) = j_i(0)$. The same argument also works for the maximal ideal $(J-1728)$ provided that $j_0=1728$ is in $k^{\text{ord}}$. This completes the proof of (i).


  Now we prove (ii). Since $J=0=1728$ is supersingular for $p=2,3$, we may assume that $p \geq 5$. Let $E$ be any ordinary elliptic curve over $k$
with $j(E)=j_0 \in k$, and let $\mathbb{E}$ be its canonical lifting over $W(k)$. Let $\Omega$ be a fixed algebraic closure of the field of fractions of $W(k)$. 
By \cite[\textsection{V}.3]{Messing} we have that 
\begin{eqnarray}
 \text{End}(\mathbb{E}) \longisomto \text{End}(E) \nonumber
\end{eqnarray}
via the reduction modulo $p$ map. Now take any $p$ such that $j_0=0$ is an ordinary $j$-value in $k$. 
Then the automorphism group of $E/k$, Aut$_{k}(E)$ has order 6 and by the above isomorphism we also have that
Aut$_{\Omega}(\mathbb{E}\otimes \Omega)$ has order at least 6. But this can happen if and only if $j(\mathbb{E})=0$, i.e. in Witt vector notation
$\Theta(0)=(0,0,0...)$. Similarly if $E/k$ with $j(E)=j_0=1728$ is ordinary for some $p$, then we have that Aut$_{\bar{k}}(E)$ has order 4. So that $j(\mathbb{E})=1728$.
\end{proof}

\end{document}